\documentclass[preprint,12pt]{amsart}

\usepackage[margin=1.5in]{geometry}

\def\Z{{\mathbb{Z}}}
\def\Q{{\mathbb{Q}}}
\def\R{{\mathbb{R}}}
\def\N{{\mathbb{N}}}
\def\C{{\mathbb{C}}}
\usepackage{amsmath}
\usepackage{amsfonts}
\usepackage{amssymb}
\usepackage{xypic}
\usepackage{longtable}
\usepackage{amsthm}

\theoremstyle{plain}
\newtheorem{theorem}{Theorem}[section]
\newtheorem{corollary}[theorem]{Corollary}
\newtheorem{lemma}[theorem]{Lemma}
\newtheorem{proposition}[theorem]{Proposition}
\theoremstyle{definition}

\newtheorem{example}[theorem]{Example}
\newtheorem{rem}[theorem]{Remark}

\makeatletter
\def\ps@pprintTitle{%
  \let\@oddhead\@empty
  \let\@evenhead\@empty
  \let\@oddfoot\@empty
  \let\@evenfoot\@oddfoot
}
\makeatother

\title[On a numerical characterization of non-simple ppavs]{On a numerical characterization of non-simple principally polarized abelian varieties}
\author{Robert Auffarth}
\thanks{The author was partially funded by Fondecyt grant No. 3150171}
\address{Departamento de Matem\'aticas, Universidad de Chile, Las Palmeras 3425,\~{N}u\~{n}oa, Santiago, Chile}
\email{rfauffar@mat.puc.cl}
\date{}

\begin{document}

\maketitle

\begin{abstract}
To every abelian subvariety of a principally polarized abelian variety $(A, \mathcal{L})$ we canonically associate a numerical class in the N\'eron-Severi group of $A$. We prove that these classes are characterized by their intersection numbers with $ \mathcal{L}$; moreover, the cycle class induced by an abelian subvariety in the Chow ring of $A$ modulo algebraic equivalence can be described in terms of its numerical divisor class. Over the field of complex numbers, this correspondence gives way to an explicit description of the (coarse) moduli space that parametrizes non-simple principally polarized abelian varieties with a fixed numerical class.
\end{abstract}

\vspace{1cm}

\section{Introduction}
\label{sec:introduction}

For many years, mathematicians have been interested in non-simple abelian varieties. Although the general abelian variety is simple, those that contain non-trivial abelian subvarieties appear frequently in nature. For example, if $C$ and $C'$ are smooth projective curves with Jacobians $J$ and $J'$, respectively, and $f:C\to C'$ is a finite morphism, then the pullback $f^*$ induces a homomorphism $f^*:J'\to J$ that has finite kernel. In particular, if $g'$ denotes the genus of $C'$, then $J$ contains an abelian subvariety of dimension $g'$. Non-simple abelian varieties have also been studied in the context of group actions. If $G$ is a group that acts on a polarized abelian variety $(A, \mathcal{L})$ (that is, $G$ acts on $A$ by regular morphisms and $g^* \mathcal{L}$ is numerically equivalent to $ \mathcal{L}$ for every $g\in G$), then irreducible representations of $G$ give way to abelian subvarieties of $A$. This approach to studying non-simple abelian varieties has been very successful and has produced a copious amount of interesting examples and families (see \cite{LR} for the general theory of group actions on abelian varieties and how they define abelian subvarieties, as well as \cite{CR}, \cite{GAR}, \cite{Paulhus} and \cite{Rojas}).

Our point of view for studying the simplicity of abelian varieties is motivated by Kani's work on abelian surfaces in \cite{Kani}. If $(A, \mathcal{L})$ is a principally polarized abelian surface, then an elliptic subgroup $E\leq A$ can be seen as a divisor, and it thus naturally induces a numerical class in $\mbox{NS}(A)$, the N\'eron-Severi group of $A$. Kani proved the following:

\begin{theorem}[Kani]
The map $E\mapsto[E]$ that takes an elliptic curve on a principally polarized abelian surface $(A,\mathcal{L})$ to its corresponding numerical class induces a bijection between
\begin{enumerate}
\item elliptic subgroups $E\leq A$ with $(E\cdot\mathcal{L})=d$
\item primitive numerical classes $\alpha\in\mbox{NS}(A)$ such that $(\alpha^2)=0$ and $(\alpha\cdot\mathcal{L})=d$.
\end{enumerate}
\end{theorem}

Here primitive means that the class is not a non-trivial multiple of another class. When the dimension of $A$ is greater than 2, this technique obviously does not work for studying elliptic curves on $A$. However, if we wish to study codimension 1 abelian subvarieties, a similar technique can be used that was studied in \cite{Auff}. Concretely, if $Z\leq A$ is a codimension 1 abelian subvariety on a principally polarized abelian variety (ppav) of dimension $n$, then the class $[Z]$ it defines in $\mbox{NS}(A)$ is primitive and satisfies 
$$\begin{array}{ll}(Z\cdot \mathcal{L}^{n-1})>0&\\(Z^r\cdot \mathcal{L}^{n-r})=0,&r\geq2.\end{array}$$
Moreover, these conditions completely characterize those numerical classes that come from abelian subvarieties of codimension 1. If we wish to study elliptic curves on $A$, then we can use the Poincar\'e Irreducibility Theorem that says that for every abelian subvariety $X\leq A$ there exists a complementary abelian subvariety $Y$ (that can be canonically defined using $ \mathcal{L}$) such that the addition map $X\times Y\to A$ is an isogeny. In particular, if $E\leq A$ is an elliptic subgroup, then its complementary abelian subvariety $Z_E\leq A$ is of codimension 1, and we obtain a function
$$E\mapsto[Z_E]\in\mbox{NS}(A).$$
The following theorem was proven in \cite{Auff}:

\begin{theorem}[Auffarth]
The map $E\mapsto[Z_E]$ induces a bijection between
\begin{enumerate}
\item elliptic subgroups $E\leq A$ with $(E\cdot\mathcal{L})=d$
\item primitive classes $\alpha\in\mbox{NS}(A)$ such that 
$$(\alpha^r\cdot \mathcal{L}^{n-r})=\left\{\begin{array}{ll}(n-1)!d&r=1\\0&r\geq2\end{array}\right.$$
\end{enumerate}
\end{theorem}

The purpose of this paper is to generalize the previous theorems and characterize \emph{all} abelian subvarieties numerically. Assuming $(A, \mathcal{L})$ to be a ppav, consider the following two well-known facts:

\begin{enumerate}
\item The N\'eron-Severi group is isomorphic to the group of endomorphisms on $A$ that are fixed by the Rosati involution
\item For every abelian subvariety $X\leq A$ there exists a \emph{norm endomorphism} $N_X\in\mbox{End}(A)$ that characterizes $X$ (see below for details).
\end{enumerate}

One pleasant fact about the norm endomorphism of $X$ is that it is fixed by the Rosati involution, and therefore via the isomorphism from $(i)$ we can associate to $X$ a numerical class $\delta_X\in\mbox{NS}(A)$.

We recall that the \emph{exponent} of an abelian subvariety $X$ is the exponent of the finite group $\ker(\phi_{ \mathcal{L}|_X})$, where for a line bundle (or numerical class) $M$ on an abelian variety $B$, $\phi_M:B\to\mbox{Pic}^0(B)= B^\vee$ denotes the morphism $x\mapsto t_x^*M\otimes M^{-1}$, where $t_x$ is translation by $x$. 

It turns out that if $d$ is the exponent of $X$, then $\delta_X=\frac{1}{d}[N_X^* \mathcal{L}]$, where $N_X$ is the norm endomorphism of $X$ (see Lemma~\ref{importantlemma} or \cite[Exercise 5.6.15]{BL}). When $X$ is an elliptic curve, $\delta_X$ coincides with the class of the complementary abelian subvariety of $X$ and therefore the assignment $X\mapsto\delta_X$ generalizes the technique found in \cite{Auff}. The principal theorems established in \cite{Kani} and \cite{Auff} can be generalized to:

\begin{theorem}
The map $X\mapsto \frac{1}{d}[N_X^* \mathcal{L}]$ gives a bijection between
\begin{enumerate}
\item abelian subvarieties of dimension $u$ and exponent $d$
\item primitive classes $\alpha\in\mbox{NS}(A)$ that satisfy
$$(\alpha^r\cdot \mathcal{L}^{n-r})=\left\{\begin{array}{ll}(n-r)!r!\binom{u}{r}d^r&1\leq r\leq u\\0&u+1\leq r\leq n\end{array}\right.$$
\end{enumerate}
Moreover, if $Y$ is the complementary abelian subvariety of $X$ in $A$ and $[Y]$ is its class in the Chow ring modulo algebraic equivalence, we have the equality
$$\delta_X^u=\frac{u!(n-u)!d^u}{(\mathcal{L}^{n-u}\cdot Y)}[Y].$$
\end{theorem}

This theorem has several pleasant consequences, such as the fact that the N\'eron-Severi group of a ppav along with its intersection pairing completely determine the abelian subvarieties that appear in $A$. This idea was already present in Bauer's work \cite{Bauer}, where he shows that non-trivial abelian subvarieties can be detected numerically. We mention the relation between his work and ours in Remark~\ref{Bauer}.

When working over the complex numbers, we show that this correspondence is quite explicit, and show how it lets us describe the coarse moduli space of non-simple ppavs. In particular, we relate the above theory to a moduli construction done by Debarre \cite{Deb} that describes the moduli space of non-simple ppavs. The complex case for $n=2$ was studied analytically by Humbert \cite{Humbert} (this is where the famous Humbert surfaces come from). The results of Section 3 substantially generalize Humbert's results.\\

\noindent\textit{Acknowledgements:} I would like to thank Anita Rojas for many helpful discussions.

\section{Abelian subvarieties and numerical classes}

Let $(A,\mathcal{L})$ be a principally polarized abelian variety (ppav) of dimension $n$ defined over an algebraically closed field $k$, where $\mathcal{L}$ is a line bundle on $A$ with $h^0(A,L):=\dim H^0(A,\mathcal{L})=1$. The \emph{N\'eron-Severi group} of $A$ is the finitely generated abelian group $\mbox{NS}(A):=\mbox{Pic}(A)/\mbox{Pic}^0(A)$.
 
If $M\in\mbox{Pic}(A)$, we will denote by $\phi_M$ the morphism $A\to A^\vee$ where $x\mapsto t_x^*M\otimes M^{-1}$. This morphism depends only on the algebraic equivalence class of $M$, and so we can just as easily define the morphism $\phi_\alpha$ for any class $\alpha\in\mbox{NS}(A)$. Via $ \mathcal{L}$, we identify $A$ with $A^\vee$, and so we will consider $\phi_M$ to be an endomorphism of $A$. We denote by $\sigma^\dagger$ the Rosati involution of an endomorphism $\sigma$ with respect to $ \mathcal{L}$; that is, $\sigma^\dagger:=\sigma^\vee\sigma$ where $\sigma^\vee$ is the dual morphism of $\sigma$. It is well-known that $\mbox{NS}(A)\simeq\mbox{End}^s(A)$, where $\mbox{End}^s(A)$ denotes the group of endomorphisms of $A$ fixed by $\dagger$.

Let $X\leq A$ be an abelian subvariety of $A$, and let $N_X:=j\psi_{ \mathcal{L}|_X}j^\vee$ be the norm endomorphism associated to $X$, where $j$ is the inclusion $X\hookrightarrow A$ and $\psi_{ \mathcal{L}|_X}$ is the unique isogeny $X^\vee\to X$ such that $\psi_{ \mathcal{L}|_X}\phi_{ \mathcal{L}|_X}$ is multiplication by the exponent of the finite group $\ker\phi_{ \mathcal{L}|_X}$ (the restriction of $ \mathcal{L}$ to $X$ gives an ample divisor on $X$, and so $\ker\phi_{ \mathcal{L}|_X}$ is finite). We define the \emph{exponent of }$X$ to be the exponent of this group. Given an abelian subvariety $X$, let $Y$ be its \emph{complementary abelian subvariety}; that is, $Y$ is the connected component of $\ker N_X$. This is an abelian subvariety of $A$ and the addition map $X\times Y\to A$ is an isogeny.

By \cite{Milne}, Theorem 10.9, if $f$ is an endomorphism of $A$, then there exists a unique monic polynomial $P_f(t)\in\Z[t]$ of degree $2n$ such that for every $m\in\Z$, $P_f(m)=\deg(f-m)$. Moreover, if $p$ is a prime that doesn't divide $\mbox{char}(k)$, then $P_f(t)$ is the characteristic polynomial of the action of $f$ on
$$V_pA=(T_pA)\otimes_{\Z_p}\Q_p,$$ where $A[p^l]$ denotes the group of $p^l$-torsion points of $A$, $T_p:=\lim_{\leftarrow}A[p^l]$ is the $p$-adic Tate module of $A$, $\Z_p$ denotes the ring of $p$-adic integers and $\Q_p$ is the field of $p$-adic numbers. We will call this polynomial the \emph{characteristic polynomial} of $f$.

Let $P_X(t)$ denote the characteristic polynomial of $N_X$. We see that $N_X^2-dN_X=0$, where $d$ is the exponent of $X$, and so the roots of $P_X(t)$ are $0$ and $d$.

\begin{proposition}
If $X\leq A$ is an abelian subvariety of dimension $u$ and exponent $d$, then
$$P_X(t)=t^{2n-2u}(t-d)^{2u}.$$
\end{proposition}
\begin{proof}
Let $Y$ be the complementary abelian subvariety of $X$ in $A$ (defined by $ \mathcal{L}$). We have the following commutative diagram:

\centerline{\xymatrix{A\ar[rr]^{(N_X,N_Y)}\ar[d]_{N_X}&&X\times Y\ar[d]^{g}\\
A\ar[rr]^{(N_X,N_Y)}&&X\times Y}}
where $g=\left(\begin{array}{cc}d\mbox{id}_X&0\\0&0\end{array}\right)$. Now $(N_X,N_Y)$ is an isogeny, and so
$$N_X=(N_X,N_Y)^{-1}\left(\begin{array}{cc}d\mbox{id}_X&0\\0&0\end{array}\right)(N_X,N_Y)$$
in $\mbox{End}_\Q(A)$. In particular,
$$\mbox{tr}(N_X)=\mbox{tr}\left(\begin{array}{cc}d\mbox{id}_X&0\\0&0\end{array}\right),$$
where $tr$ denotes the trace function on $\mbox{End}(A)\hookrightarrow\mbox{End}(V_p)$. Let $p$ be a prime such that $p\neq\mbox{char}(k)$ and such that $A[p]\not\subseteq X\cap Y$ ($p$ exists since $X\cap Y$ is finite). In this case, $A[p]\simeq X[p]\oplus Y[p]$, and so
$$V_p A\simeq (T_pZ\oplus T_p Y)\otimes_{\Z_p}\Q_p\simeq V_pZ\oplus V_pY.$$
It is then obvious that $g$ acts on $V_pA$ as $\left(\begin{array}{cc}d\mbox{id}_X&0\\0&0\end{array}\right)$, and so
$$\mbox{tr}(N_X)=\mbox{tr}\left(\begin{array}{cc}d\mbox{id}_X&0\\0&0\end{array}\right)=2ud$$
(since $\dim_{\Q_p}(T_pZ)\otimes_{\Z_p}\Q_p=2u$). Note that this argument was adapted from the argument of \cite{BL} Corollary 5.3.10.

From our previous discussion, we have that
$$P_X(t)=t^{2n-r}(t-d)^{r}=\sum_{m=0}^{r}\binom{r}{m}(-d)^mt^{2n-m}.$$
Moreover, we have that $\mbox{tr}(N_X)$ is precisely $-1$ times the coefficient of $t^{2n-1}$. Putting everything together, we get
$$2ud=\mbox{tr}(N_X)=d\binom{r}{1}=dr$$
and so $r=2u$. 
\end{proof}

The key observation that allows us to study abelian subvarieties using numerical classes is the following:\\

\noindent\textbf{Key Observation:} Since $N_X$ is a symmetric endomorphism, there exists a unique numerical class $\delta_X\in\mbox{NS}(A)$ such that $N_X=\phi_{\delta_X}$.\\

This association then defines a function 
$$\Delta:\{\mbox{abelian subvarieties of }A\}\to\mbox{NS}(A)$$
which we can write explicitly:

\begin{lemma}\label{importantlemma}
$\delta_X=\frac{1}{d}[N_X^* \mathcal{L}]$.
\end{lemma}
\begin{proof}
By identifying $A$ with $A^\vee$, we have
$$\phi_{N_X^* \mathcal{L}}=N_X^\vee\phi_ \mathcal{L} N_X=N_X^2=dN_X=\phi_{d\delta_X}.$$
The N\'eron-Severi group of an abelian variety is torsion free, and so the equality follows.
\end{proof}

Since an abelian subvariety is determined by its norm endomorphism, we see that $\Delta$ is injective. Our goal is to study the image of $\Delta$. We know by the properties of norm endomorphisms in \cite{BL}, page 125 (the arguments there are valid over any algebraically closed field), that
$$N_X+N_Y=d\mbox{id},$$
where $d$ is the exponent of $X$. In terms of numerical classes, we get that
$$\delta_X=d[ \mathcal{L}]-\delta_Y,$$
where $[ \mathcal{L}]$ denotes the numerical class of $ \mathcal{L}$. In particular, $[\delta_X]=-[\delta_Y]$ in $\mbox{NS}(A)/\Z[ \mathcal{L}]$ and so the natural function between abelian subvarieties of dimension $u$ and abelian subvarieties of codimension $u$ can be seen as the inverse homomorphism.

\begin{example}
We have that $\delta_{\{0\}}=0$ and $\delta_A=[ \mathcal{L}]$.
\end{example}

\begin{example}
If $X$ is an elliptic subgroup of $A$ (that is, an abelian subvariety of dimension 1), then $\delta_X$ corresponds to the numerical class of the complementary abelian subvariety of $X$ (this can be easily shown by hand, or by using the next theorem). In particular, $\Delta$ generalizes the function described in \cite{Auff}.
\end{example}

\begin{rem}\label{Bauer}
In \cite{Bauer}, Bauer showed that if $X$ is an abelian subvariety of $A$, then
$$\mbox{sup}\{t\in\R:\mathcal{L}-tN_X^*\mathcal{L}\mbox{ is nef}\}=\frac{1}{d^2}\in\Q$$
where $d$ is the exponent of $X$ (here we use additive notation for the line bundles). Moreover, Bauer shows that this characterizes non-simple abelian varieties; that is, if $(A,\mathcal{L})$ is a simple polarized abelian variety and $M\in\mbox{Pic}(A)$, then
$$\sup\{t\in\R:\mathcal{L}-tM\mbox{ is nef}\}$$
is always irrational. When $\mathcal{L}$ is a principal polarization, since 
$$[\mathcal{L}]-\frac{1}{d^2}[N_X^*\mathcal{L}]=[\mathcal{L}]-\frac{1}{d}\delta_X=\frac{1}{d}\delta_Y$$
where $Y$ is the complementary abelian subvariety of $X$, Bauer's result essentially says that if $\delta_A$ can be written as $\delta_A=\alpha+t\beta$ in $\mbox{NS}(A)\otimes\R$ for $\alpha$ nef and $t$ maximal, then $t$ is rational if and only if $\alpha=\frac{1}{d}\delta_X$ and $\beta=d\delta_Y$ for certain non-trivial complementary abelian subvarieties $X,Y\subseteq A$. This shows that the existence of certain numerical classes guarantees that $A$ contains a non-trivial abelian subvariety. We can make this more precise with our theory in what follows.
\end{rem}

Our main theorem shows that abelian subvarieties are characterized by the intersection numbers of their numerical classes.

\begin{theorem}\label{numchar} The map $\Delta:X\mapsto \frac{1}{d}[N_X^* \mathcal{L}]$ gives a bijection between
\begin{enumerate}
\item abelian subvarieties of dimension $u$ and exponent $d$
\item primitive classes $\alpha\in\mbox{NS}(A)$ that satisfy
$$(\alpha^r\cdot \mathcal{L}^{n-r})=\left\{\begin{array}{ll}(n-r)!r!\binom{u}{r}d^r&1\leq r\leq u\\0&u+1\leq r\leq n\end{array}\right.$$
\end{enumerate}
Moreover, if $Y$ is the complementary abelian subvariety of $X$ in $A$, in the Chow ring modulo algebraic equivalence $\frak{A}^*(A)$ we have the equality
$$\delta_X^u=\frac{u!(n-u)!d^u}{(\mathcal{L}^{n-u}\cdot Y)}[Y].$$
\end{theorem}

Before proving the theorem we prove a lemma.\\

\begin{lemma}
Let $\phi\in\mbox{End}^s(A)$ be a symmetric endomorphism with characteristic polynomial of the form $P_\phi(t)=t^{2n-2u}(t-d)^{2u}$. Then the minimal polynomial of $\phi$ is $M_\phi(t)=t(t-d)$.
\end{lemma}
\begin{proof}
Let $\Q[\phi]$ be the (commutative) subalgebra of $\mbox{End}_\Q(A)$ generated by $1$ and $\phi$, and let $T_\phi\in\mbox{End}(\Q[\phi])$ be multiplication by $\phi$. We know that $H:(f,g)\mapsto\mbox{Tr}(f g^\dagger)$ is a positive definite symmetric bilinear form, and we see that for all $f,g\in\Q[\phi]$
$$H(T_\phi(f),g)=\mbox{Tr}(\phi f g^\dagger)=\mbox{Tr}(f  g^\dagger\phi^\dagger)=\mbox{Tr}(f (\phi g)^\dagger)=H(f,T_\phi(g)),$$
and so $T_\phi$ is self-adjoint with respect to $H$. In particular, $T_\phi$ is diagonalizable on $\R[\phi]=\Q[\phi]\otimes\R$, and so its minimum polynomial splits as the product of distinct linear factors:
$$M_{T_\phi}(t)=\prod_{i=1}^r(t-\lambda_i)$$
where $r=\dim_\Q\Q[\phi]$ and $\lambda_i\in\R$. Evaluating in $T_\phi(1)$, we get that
$$0=\prod_{i=1}^r(\phi-\lambda_i\mbox{id}).$$
Therefore the minimal polynomial of $\phi$ divides $M_{T_\phi}$, and by our hypothesis on the characteristic polynomial of $\phi$, we have that $M_\phi(t)$ must be $t(t-d)$.

\end{proof}

\begin{proof}[Proof of Theorem~\ref{numchar}]
If $X$ is an abelian subvariety of dimension $u$ and exponent $d$, by Riemann-Roch we have 
$$P_{N_X}(m)=m^{2n-2u}(m-d)^{2u}=\deg(N_X-m)=\deg(m-N_X)=\left(\frac{(m[ \mathcal{L}]-\delta_X)^n}{n!}\right)^2.$$
Therefore for $m\gg0$, since $m[ \mathcal{L}]-\delta_X$ is an ample class we obtain that its self intersection is positive, and so
$$m^{n-u}(m-d)^u=\frac{(m[ \mathcal{L}]-\delta_X)^n}{n!}=\frac{1}{n!}\sum_{r=0}^n\binom{n}{r}(-1)^{r}(\delta_X^r \mathcal{L}^{n-r})m^{n-r}.$$
On the other hand,
$$m^{n-u}(m-d)^u=\sum_{r=0}^u\binom{u}{r}(-1)^rd^rm^{n-r}.$$
By comparing the coefficients of these two polynomials we obtain the desired intersection numbers.

Now let $\alpha$ be a primitive class that has the intersection numbers above. By our previous analysis, we have that $P_{\phi_\alpha}(t)=t^{2n-2u}(t-d)^{2u}$. Since $\phi_\alpha$ is a symmetric endomorphism, by the previous lemma we get that $\phi_\alpha^2=d\phi_\alpha$. Moreover $\phi_\alpha$ is primitive since $\alpha$ is, and so by the Norm-endomorphism criterion in \cite{BL} page 124 (the arguments there work over any algebraically closed field), we get that $\phi_\alpha=N_{\mbox{Im}(\phi_\alpha)}$ and $\mbox{Im}(\phi_\alpha)$ is an abelian subvariety of dimension $u$ and exponent $d$.

Since $(\ker\phi_{\delta_X})_0=Y$ (the connected component of $\ker\phi_{\delta_X}$ containing 0), we have that there exists a line bundle $\tilde{\mathcal{L}}_X$ on $A/Y$ such that the numerical class of $\mathcal{L}_X:=p_Y^*\tilde{\mathcal{L}}_X$ is $\delta_X$, where $p_Y:A\to A/Y$ is the natural projection. We see that if we intersect the numerical class of $\tilde{\mathcal{L}}_X$ $u$ times (as algebraic cycles),  we obtain a finite number of points. Since any two points are algebraically equivalent on an abelian variety, we have that $\delta_X^u= m_Yp_Y^*[\{0\}]$ for some integer $m_Y\geq0$ in the Chow ring of $A$ modulo algebraic equivalence. Now, $p_Y^*[\{0\}]=[Y]$, and so we get that
$$\delta_X^u=m_Y[Y]$$
for some integer $m_Y$.

We get that
$$(n-u)!^2u!^2d^{2u}=(\delta_X^u\cdot \mathcal{L}^{n-u})^2=m_Y^2( \mathcal{L}|_Y^{n-u})^2,$$
and so
$m_Y=\frac{u!(n-u)!d^u}{(\mathcal{L}^{n-u}\cdot Y)}.$ 
\end{proof}

The following corollary was of course implicit in \cite{Bauer}, but can be seen more explicitly with the previous theorem.

\begin{corollary}
Simplicity for an abelian variety is a numerical property. In particular, if $(A, \mathcal{L})$ and $(B,\Xi)$ are two ppavs of dimension $n$ such that there exists an isomorphism $\phi:\mbox{NS}(A)\to\mbox{NS}(B)$ that preserves the intersection pairing, then $\phi$ induces a bijection between abelian subvarieties of $A$ and abelian subvarieties of $B$ that preserves dimension and exponents.
\end{corollary}

\begin{corollary}
Let $\overline{\Sigma}(A)$ be the graded subring of $\frak{A}^*(A)$ generated by all cycle classes of abelian subvarieties and their numerical divisor classes, and let $\overline{\Sigma}(A)_\Q:=\overline{\Sigma}(A)\otimes_\Z\Q$. Then $\overline{\Sigma}(A)_\Q$ is a finite-dimensional vector space over $\Q$.
\end{corollary}

\begin{proof}
Since $\mbox{NS}(A)$ is finitely generated, the $\Z$-module generated by the $\delta_X$ is generated by certain classes $\delta_{X_1},\ldots,\delta_{X_m}$ (recall that $\delta_A=[ \mathcal{L}]$). By Theorem~\ref{numchar}, for every abelian subvariety $X$ of dimension $u$ and exponent $d$ we have
$$(d[ \mathcal{L}]-\delta_X)^{n-u}=\frac{(n-u)!u!d^{n-u}}{(\mathcal{L}^u\cdot X)}[X].$$
Since $\delta_X$ and $[ \mathcal{L}]$ are linear combinations of the $\delta_{X_i}$ (remember that $[ \mathcal{L}]=\delta_A$), we obtain that $\overline{\Sigma}(A)_\Q=\Q[\delta_{X_1},\ldots,\delta_{X_m}]$. Now since each $\delta_{X_i}$ is integral over $\Q$ (they are nilpotent classes after all), $\Q[\delta_{X_1},\ldots,\delta_{X_m}]$ is a finitely generated $\Q$-module.
\end{proof}

\begin{corollary}
For each $i$, the $\Q$-vector space $\Sigma^i(A)_\Q$ generated by the cycle classes of codimension $i$ abelian subvarieties is finite dimensional.
\end{corollary}

One question that seems interesting is: If $\Omega$ is the group generated by the classes $\delta_X$, what is its rank in comparison to the Picard number of $A$? If $A$ is simple then the question is uninteresting, since by Proposition 5.5.7 in \cite{BL} the Picard number of a simple complex ppav can vary between 1 and $3n/2$, but $\Omega=\langle[\mathcal{L}]\rangle$. Is there a non-trivial relation between the two ranks for non-simple ppavs?

Using intersection numbers, we can tell when an abelian subvariety of codimension 1 contains an elliptic curve:

\begin{corollary}
Let $E,Z\leq A$ be abelian subvarieties of dimensions $1$ and $n-1$ and exponents $d_E$ and $d_Z$, respectively, and let $\delta_E$ and $\delta_Z$ be their respective numerical classes. Then $E$ is contained in $Z$ if and only if 
$$(\delta_E\cdot\delta_Z\cdot\mathcal{L}^{n-2})=d_Ed_Z(n-2)!(n-2).$$
\end{corollary}

\begin{proof}
We have that $E\subseteq Z$ in this case if and only if $[E][Z]=0$ in $\frak{A}^*(A)$. By Theorem~\ref{numchar}, this is equivalent to $\delta_{C(E)}^{n-1}\delta_{C(Z)}=0$, where $C(E)$ and $C(Z)$ denote the complementary abelian subvarieties of $E$ and $Z$, respectively. Since $\delta_{C(E)}=d_E[\mathcal{L}]-\delta_E$ (and similarly for $Z$), this is equivalent to
$$(d_E[\mathcal{L}]-\delta_E)^{n-1}(d_Z[\mathcal{L}]-\delta_Z)=0.$$
After expanding this expression we obtain the desired result.
\end{proof}

It would be nice to have a numerical criterion for when an abelian subvariety is non-simple in general. 

For $n=2$, Kani showed in \cite{Kani} that elliptic subgroups on principally polarized abelian surfaces can be described by the quadratic form $q(\alpha)=(\alpha\cdot\mathcal{L})^2-2(\alpha^2)$ on $\mbox{NS}(A)/\Z[\mathcal{L}]$. The author showed in \cite{Auff} that for $n\geq3$ the analogous form $q(\alpha)=(\alpha\cdot\mathcal{L}^{n-1})^2-n!(\alpha^2\cdot\mathcal{L}^{n-2})$ is not enough to describe elliptic subgroups, and homogeneous forms $q_2,\ldots,q_n$ were introduced on $\mbox{NS}(A)/\Z[\mathcal{L}]$ in order to classify elliptic subgroups in higher dimensions. Recall that 
$$q_r(\alpha):=-\frac{1}{(r-1)n!}((\alpha^\natural)^r\cdot\mathcal{L}^{n-r}),$$
where $\alpha^\natural=n!\alpha-(\alpha\cdot\mathcal{L}^{n-1})[\mathcal{L}]$ for $\alpha\in\mbox{NS}(A)$.  With a little work, it can be shown that the same homogeneous forms characterize, via $\Delta$, not only elliptic subgroups on $A$, but all abelian subvarieties on $A$. The generalization is as follows:

\begin{theorem}
Let $\pi:\mbox{NS}(A)\to\mbox{NS}(A)/\Z[\mathcal{L}]$ be the natural projection. Then the map $X\mapsto\pi(\frac{1}{d}[N_X^*\mathcal{L}])$ induces a bijection between
\begin{enumerate}
\item abelian subvarieties of dimension $u$ and exponent $d$ on $A$
\item primitive classes $[\alpha]\in\mbox{NS}(A)/\Z[\mathcal{L}]$ that satisfy $(\alpha\cdot\mathcal{L}^{n-1})\equiv (n-1)!ud\mbox{ (mod }n!)$ and $q_r(\alpha)=f(u,r)d^r$ for $2\leq r\leq n$, where
\end{enumerate}
\begin{eqnarray}\nonumber f(u,r)&:=&\frac{1}{r-1}\sum_{m=0}^{\min\{r,u\}}\binom{r}{m}\binom{u}{m}n!^{m-1}(n-1)!^{r-m}(n-m)!m!(-1)^{r-m+1}u^{r-m}\end{eqnarray}
\end{theorem}

The condition $(\alpha\cdot\mathcal{L}^{n-1})\equiv (n-1)!ud\mbox{ (mod }n!)$ is conjectured to be superfluous, as in the case $u=1$. We omit the proof of this theorem; it is not very different from the proof of the special case proved in \cite{Auff}.

\section{Analytic theory}

When $k=\mathbb{C}$, we have that $A\simeq\C^n/\Lambda$ for some lattice $\Lambda$. By choosing real coordinates $x_1,\ldots,x_{2n}$ on $A$ that come from a symplectic basis $\mathcal{B}:=\{\lambda_1,\ldots,\lambda_n,\mu_1,\ldots,\mu_n\}$ for $\Lambda$, $\mbox{NS}(A)$ can be canonically identified with a subgroup of $\bigwedge^2\Z^{2n}\simeq H^2(A,\Z)$ such that the class of $ \mathcal{L}$ is $\theta:=-\sum_{i=1}^ndx_i\wedge dx_{i+n}$; this is done via the first Chern class $c_1:\mbox{Pic}(A)\to H^2(A,\Z)$. It is interesting to observe that $\bigwedge^2\Z^{2n}$ no longer depends on $A$, and this gives us a good environment to study moduli of non-simple ppavs. Let $\tau$ be the period matrix for $(A, \mathcal{L})$ defined by $\mathcal{B}$ and the complex basis $\mu_1,\ldots,\mu_n$ of $\C^n$. The respective real and complex coordinates are related by the following formula:
$$\begin{pmatrix}z_1\\\vdots\\z_n\end{pmatrix}=(\tau\hspace{0.1cm}I)\begin{pmatrix}x_1\\\vdots\\x_{2n}\end{pmatrix}.$$
Given that $\mbox{NS}(A)=H^2(A,\Z)\cap H^{1,1}(A,\C)$, we obtain
$$\mbox{NS}(A)=\{\omega\in H^2(A,\Z):\omega\wedge dz_1\wedge\cdots\wedge dz_n=0\}.$$
In this context, Theorem~\ref{numchar} essentially says:

\begin{theorem}\label{main} The map $X\mapsto c_1(\frac{1}{d}N_X^* \mathcal{L})$ induces a bijection between
\begin{enumerate}
\item abelian subvarieties of dimension $u$ and exponent $d$
\item primitive differential forms $\eta\in H^2(A,\Z)$ that satisfy 
\begin{enumerate}
\item $\eta\wedge dz_1\wedge\cdots\wedge dz_n=0$ and
\item$\eta^{\wedge r}\wedge \theta^{\wedge(n-r)}=\left\{\begin{array}{ll}(n-r)!r!\binom{u}{r}d^r\omega_0&1\leq r\leq u\\0&u+1\leq r\leq n\end{array}\right.$
\end{enumerate}
where $\omega_0=(-1)^ndx_1\wedge dx_{n+1}\cdots\wedge dx_{n}\wedge dx_{2n}$.
\end{enumerate}
\end{theorem}

Observe that if we canonically identify $H^2(A,\Z)$ with $\bigwedge^2\Z^{2n}$ by using the coordinates $x_1,\ldots,x_{2n}$, condition $(b)$ in the previous theorem only depends on the differential form $\eta$ and not on the abelian variety (since we have fixed $c_1( \mathcal{L})=\theta$). For $\eta\in\bigwedge^2\Z^{2n}$, define the sets
$$\mathbb{H}_n(\eta):=\{\tau\in\mathbb{H}_n:\eta\wedge dz_1\wedge\cdots\wedge dz_n=0\}$$
$$\mathcal{A}_n(\eta):=\pi_n(\mathbb{H}_n(\eta)),$$
where $\pi_n:\mathbb{H}_n\to\mathcal{A}_n$ is the natural projection, and define
$$B(n,u,d):=\{\eta\in \bigwedge^2\Z^{2n}:\eta\mbox{ satisfies condition }(b)\mbox{ and }\mathbb{H}_n(\eta)\neq\varnothing\}.$$
One thing we can do here is fix a period matrix of a certain ppav and go over elements of $B(n,u,d)$ (with a computer, for example), checking to see if they satisfy condition $(a)$. This amounts to looking for abelian subvarieties on a fixed ppav, and this method seems to be useful for calculating examples. A second point of view is to fix an element $\eta\in B(n,u,d)$ and to search for all period matrices that satisfiy condition $(a)$. This is equivalent to looking for all ppavs with an abelian subvariety that induces the class $\eta$. In the rest of the paper, we will analyze the second point of view. We note that this second point of view gives us equations on $\mathbb{H}_n$ that describe when the period matrix of a ppav contains an abelian subvariety of fixed dimension and exponent.

If $\eta=\sum_{i<j} a_{ij}dx_i\wedge dx_j\in B(n,u,d)$, define the $2n\times 2n$ matrix
$$M_\eta:=(a_{ij})_{i,j},$$
where $a_{ji}:=-a_{ij}$. Define $J$ to be the $2n\times 2n$ matrix
$$J=\left(\begin{array}{cc}0&I\\-I&0\end{array}\right).$$

\begin{proposition}\label{exa}
Let $X\leq A$ be a $u$-dimensional abelian subvariety of exponent $d$, and let $\eta\in B(n,u,d)$ be its numerical class. Then the rational representation of $N_X$ (with respect to the symplectic basis above) is given by the matrix $JM_\eta$. In particular, $u=\frac{1}{2}\mbox{rank}(M_\eta)$ and the vector space in $\C^n$ that defines $X$ can be identified, using $\{\lambda_1,\ldots,\mu_n\}$, with the image of $JM_\eta$. Moreover, $d=-\frac{1}{u}(a_{1,n+1}+\cdots+a_{n,2n})=\frac{1}{2u}\mbox{tr}(N_X)$.
\end{proposition}

\begin{proof}
Let $[\rho_a(N_X)]$ be the matrix of the rational representation of $N_X$ with respect to the symplectic basis. Since $N_X$ is self-adjoint with respect to the positive definite Hermitian form defined by $ \mathcal{L}$, we have that
$$M_\eta=\frac{1}{d}N_X^*M_\theta=\frac{1}{d}[\rho_a(N_X)]^tM_\theta[\rho_a(N_X)]=\frac{1}{d}M_\theta[\rho_a(N_X)]^2=M_\theta[\rho_a(N_X)].$$
Now, $M_\theta=-J=J^{-1}$, and so $[\rho_a(N_X)]=JM_\eta$. The exponent can be calculated using condition $(b)$ of Theorem~\ref{main} with $r=1$.
\end{proof}

Let $S*\tau$ denote the usual action of an element $S\in\mbox{Sp}(2n,\Z)$ on $\tau\in\mathbb{H}_n$; that is,
$$\left(\begin{array}{cc}\alpha&\beta\\\gamma&\delta\end{array}\right)*\tau:=(\alpha\tau+\beta)(\gamma\tau+\delta)^{-1}.$$
The matrix $S^t$ is the rational representation of an isomorphism $A_{S*\tau}\to A_\tau$, and so if $X_\tau\leq A_\tau$ is an abelian subvariety, we define an abelian subvariety $X_{S*\tau}\leq A_{S*\tau}$ by the commutative diagram

\centerline{\xymatrix{A_{S*\tau}\ar[r]^{S^t}\ar[d]_{N_{X_{S*\tau}}}&A_\tau\ar[d]^{N_{X_\tau}}\\A_{S*\tau}&A_\tau\ar[l]_{S^{-t}}}}

\vspace{0.5cm}

This then gives us an action of $\mbox{Sp}(2n,\Z)$ on $B(n,u,d)$ where $\eta$ is sent to the form $S*\eta$ whose matrix is $M_{S*\eta}=SM_\eta S^t$. We wish to describe the orbits of this action.

Remember that the \emph{type} of an abelian subvariety $X\leq A$ is the list $(d_1,\ldots,d_u)$ such that $\ker\phi_{ \mathcal{L}|_X}\simeq\bigoplus_{i=1}^u(\Z/d_i\Z)^2$.

\begin{lemma}\label{type}
If $X\leq A=\C^n/\Lambda$ is an abelian subvariety, then its type is determined by its associated differential form.
\end{lemma}

\begin{proof}
Let $\rho_a(N_X)$ be the rational representation of $N_X$ with respect to the symplectic basis given earlier, and let $V_X$ be the $\R$-vector space generated by its columns. By the previous proposition, $V_X$ can be identified with the vector space that induces $X$. Similarly, set $\Lambda_X:=V_X\cap\Z^{2n}$; this can be seen as the lattice of $X$. We see that the type of $X$ is completely determined by the finite group $\ker\phi_{ \mathcal{L}|_X}$. After writing out the definitions, we see that
$$\ker\phi_{ \mathcal{L}|_X}\simeq\{u\in V_X:u^tJv\in\Z,\forall v\in\Lambda_X\}/\Lambda_X.$$
This group does not involve the period matrix of $A$.
\end{proof}

This lemma means that we can define the \emph{type} of an element in $B(n,u,d)$ in the obvious way. Proposition~\ref{exa} and Lemma~\ref{type} give us a concrete method of working with non-simple ppavs.

\begin{example}

Consider the differential form $\eta_0=dx_3\wedge dx_8-dx_3\wedge dx_7 + dx_2\wedge dx_5 -dx_2\wedge dx_6 + dx_1\wedge dx_6 + dx_4\wedge dx_7 -dx_1\wedge dx_5 -dx_4\wedge dx_8\in\bigwedge^2\Z^8$. This satisfies condition $(ii)$ of Theorem~\ref{main} with $n=4$, $u=2$ and $d=2$. We have

$$N_{\eta_0}:=JM_\eta=\left(\begin{array}{cccc}A&0&0&0\\0&A&0&0\\0&0&A&0\\0&0&0&A\end{array}\right)$$

where

$$A=\left(\begin{array}{cc}1&-1\\-1&1\end{array}\right).$$

Moreover, a matrix $\tau\in\frak{H}_4$ satisfies $\eta_0\wedge dz_1\wedge dz_2\wedge dz_3\wedge dz_4=0$ if and only if it is of the form

$$\tau=\left(\begin{array}{cccc}\tau_1&\tau_2&\tau_3&\tau_4\\\tau_2&\tau_1&\tau_4&\tau_3\\\tau_3&\tau_4&\tau_5&\tau_6\\\tau_4&\tau_3&\tau_6&\tau_5\end{array}\right)$$

Let $(A_\tau,\Theta_\tau)$ be the ppav corresponding to the matrix $\tau$ above, and let $X_\tau$ be the abelian surface that induces the class $\eta_0$. Then using Proposition~\ref{exa}, the vector space that defines $X_\tau$ is generated over $\C$ by the vectors

$$\begin{pmatrix}1\\-1\\0\\0\end{pmatrix},\begin{pmatrix}0\\0\\1\\-1\end{pmatrix}$$

and its lattice is spanned over $\Z$ by the vectors

$$\begin{pmatrix}\tau_1-\tau_2\\\tau_2-\tau_1\\\tau_3-\tau_4\\\tau_4-\tau_3\end{pmatrix},\begin{pmatrix}\tau_3-\tau_4\\\tau_4-\tau_3\\\tau_5-\tau_6\\\tau_6-\tau_5\end{pmatrix},\begin{pmatrix}1\\-1\\0\\0\end{pmatrix},\begin{pmatrix}0\\0\\1\\-1\end{pmatrix}.$$

The complementary abelian variety of $X_\tau$ in $A_\tau$ can be found using the same method. By taking the alternating form $-M_{\eta_0}$, we can see that $X_\tau$ is an abelian subvariety of type $(2,2)$. \qed\\

\end{example}

Let $D=(d_1,\ldots,d_u)\in\Z_{>0}^u$ be such that $d_i\mid d_{i+1}$, $\tilde{D}=(1,\ldots,1,d_1,\ldots,d_u)$ (of length $n-u$) and let $K(D):=(\Z/d_1\Z\times\cdots\times\Z/d_u\Z)^2$. We see that $K(D)$ has a symplectic pairing $e^D$, where 
$$e^D(e_j,e_k)=\left\{\begin{array}{ll}e^{-2\pi i/d_j}&k=u+j\\e^{2\pi i/d_j}&j=u+k\\1&\mbox{otherwise}\end{array}\right.,$$
and $e_j$ denotes the $j$th canonical vector. Let $(X,\mathcal{L}_X)$ be a polarized abelian variety of type $D$, and define 
$$K(\mathcal{L}_X):=\{x\in X:t_x^*\mathcal{L}_X\simeq\mathcal{L}_X\}.$$
It is well-known that $K(\mathcal{L}_X)\simeq K(D)$ (non-canonically), and $K(\mathcal{L}_X)$ has a symplectic pairing $e^{\mathcal{L}_X}:K(\mathcal{L}_X)^2\to\mathbb{G}_m$ induced by the polarization.

We will now recall a construction done by Debarre in \cite{Deb}. Let $\mathcal{A}_u(D)$ denote the coarse moduli space parametrizing triples $(X,\mathcal{L}_X,f)$ where $(X,\mathcal{L}_X)$ is a polarized abelian variety of type $D$ and $f:K(\mathcal{L}_X)\to K(D)$ is a symplectic isomorphism (that is, preserves the form described above). Let $\mathcal{A}_{u,n-u}^D$ denote the subvariety of $\mathcal{A}_n$ consisting of ppavs that contain an abelian subvariety of dimension $u$ and type $D$ and let $\epsilon$ be the anti-symplectic involution of $K(D)=K(\tilde{D})$ such that if $x,y\in\Z/d_1\Z\times\cdots\times\Z/d_u\Z$, then $\epsilon(x,y)=(y,x)$. Debarre presents the morphism
$$\Phi_{u,n-u}(D):\mathcal{A}_u(D)\times\mathcal{A}_{n-u}(\tilde{D})\to\mathcal{A}_{u,n-u}^D$$
where $((X,\mathcal{L}_X,f),(Y,\mathcal{L}_Y,g))\mapsto (X\times Y)/\Gamma_{g^{-1}\epsilon f}$, and $\Gamma_{g^{-1}\epsilon f}$ denotes the graph of $g^{-1}\epsilon f$. By descent theory for abelian varieties, $(X\times Y)/\Gamma_{g^{-1}\epsilon f}$ has a unique principal polarization $\mathcal{L}$ (that is, a line bundle unique up to translation) such that its pullback to $X\times Y$ is $\mathcal{L}_X\boxtimes\mathcal{L}_Y$. Debarre shows that $\Phi_{u,n-u}(D)$ is surjective, and in particular the space of ppavs that contain an abelian subvariety of fixed type is irreducible.

Let $\mbox{Sp}(D)$ denote the group of all symplectic automorphisms of $K(D)$. We notice that $\mbox{Sp}(D)$ acts on $\mathcal{A}_u(D)\times\mathcal{A}_{n-u}(\tilde{D})$ by 
$$h\cdot((X,\mathcal{L}_X,f),(Y,\mathcal{L}_Y,g))=((X,\mathcal{L}_X,hf),(Y,\mathcal{L}_Y,\epsilon h\epsilon^{-1}g)).$$ 
We assume the following is known by the experts (Kani mentions this in \cite{Kani} for the case $u=1$ and $n=2$), but because of the lack of an adequate reference we prove it here:

\begin{theorem}\label{moduli}
If $u<n/2$, the morphism $\Phi_{u,n-u}(D)$ induces an isomorphism 
$$(\mathcal{A}_u(D)\times\mathcal{A}_{n-u}(\tilde{D}))/\mbox{Sp}(D)\to\mathcal{A}_{u,n-u}^D.$$ 
If $n$ is even and $u=n/2$ then $\Phi_{n/2,n/2}(D)$ induces an isomorphism
$$(\mathcal{A}_{n/2}(D)\times\mathcal{A}_{n/2}(D))/(\Z/2\Z\ltimes\mbox{Sp}(D))\to\mathcal{A}_{u,n-u}^D$$
where $\Z/2\Z$ interchanges the factors.
\end{theorem}
\begin{proof} Assume $u<n/2$. Then if $h\in\mbox{Sp}(D)$, $(\epsilon h\epsilon g)^{-1}\epsilon hf=g^{-1}\epsilon f$ and so this action obviously permutes the fibers of $\Phi_{u,n-u}(D)$. If $((X,\mathcal{L}_X,f_i),(Y,\mathcal{L}_Y,g_i))$, $i=1,2$, are two pairs that induce the same ppav, then by descent theory for abelian varieties, the subgroups of $X\times Y$ (that is, the graphs) induced by $g_1^{-1}\epsilon f_1$ and $g_2^{-1}\epsilon f_2$ must be equal. Therefore $(f_2,g_2)=((\epsilon g_2g_1^{-1}\epsilon) f_1,(\epsilon g_2g_1^{-1}\epsilon)g_1)$ and the first part is proved.

If $n$ is even and $u=n/2$, then it is easy to see that switching the two factors leaves the fibers invariant. Moreover, via the homomorphism $\Z/2\Z\to\mbox{Aut}(\mbox{Sp}(D))$ where $1\mapsto (h\mapsto \epsilon h\epsilon)$, a proof similar to the previous one shows that two elements in any fiber differ by the action of $\Z/2\Z\ltimes\mbox{Sp}(D)$.

\end{proof}

This theorem has several pleasing consequences that allow us to better understand the differential forms associated to abelian subvarieties.

\begin{lemma}
Let $D=(d_1,\ldots,d_u)$. Then there exists a form $\eta_0$ such that $\mathcal{A}_{u,n-u}^D=\mathcal{A}_n(\eta_0)$.
\end{lemma}

\begin{proof}
Let $B'$ consist of all differential forms in $B(n,u,d)$ of type $D$. We see that
$$\mathcal{A}_{u,n-u}^D=\bigcup_{\eta\in B'}\mathcal{A}_n(\eta).$$
The irreducibility of $\mathcal{A}_{u,n-u}^D$ implies that $\mathcal{A}_{u,n-u}^D=\mathcal{A}_n(\eta_0)$ for some $\eta_0\in B'$.
\end{proof}

This lemma can be considerably strengthened:

\begin{lemma}
For all $\eta\in B(n,u,d)$ of type $D$, $\mathcal{A}_{u,n-u}^D=\mathcal{A}_n(\eta)$.
\end{lemma}

\begin{proof} Assume that $u\leq n/2$ and let $N_\eta$ be the norm matrix associated to $\eta$. Define $V_\eta$ to be the $\R$-span of the columns of $N_\eta$, $\Lambda_\eta=V_\eta\cap\Z^{2n}$, $W_\eta=\ker N_\eta$, $\Gamma_\eta=W_\eta\cap\Z^{2n}$ and $\tilde{D}$ the complementary type of $D$ (that is, $\tilde{D}$ is of length $n-u$, contains $D$ and all the rest of its entries are 1). Since $\mathcal{A}_n(\eta)\neq\varnothing$, then $N_\eta$ is the rational representation of a norm endomorphism for some ppav, and so if $E$ denotes the alternating form defined by the matrix $J$ on $\R^{2n}$, we have that $E|_{V_\eta}$ is of type $D$ and $E|_{W_\eta}$ is of type $\tilde{D}$. We notice, moreover, that $V_\eta\oplus W_\eta=\R^{2n}$, $\Lambda_\eta$ and $\Gamma_\eta$ are lattices in their respective vector subspaces and $V_\eta$ and $W_\eta$ are orthogonal with respect to $E$.

We now proceed with a construction that shows that the image of the map described in Theorem~\ref{moduli} coincides with $\mathcal{A}_n(\eta)$. Let $(X,\mathcal{L}_X)\in\mathcal{A}_D$, $(Y,\mathcal{L}_Y)\in\mathcal{A}_{\tilde{D}}$ be polarized abelian varieties (where $\mathcal{A}_D$ is the moduli space of polarized abelian varieties of type $D$), and assume that the abelian varieties can be written as complex tori $X=V/\Lambda_X$ and $Y=W/\Lambda_Y$. Let $E_X$ (resp. $E_Y$) denote the alternating form on $\Lambda_X$ (resp. $\Lambda_Y$) induced by $\mathcal{L}_X$ (resp. $\mathcal{L}_Y$). Now let 
$$\psi_X:V\to V_\eta$$
$$\psi_Y:W\to W_\eta$$
be symplectic isomorphisms that take $\Lambda_X$ (resp. $\Lambda_Y$) to $\Lambda_\eta$ (resp. $\Gamma_\eta$). This can be done, for example, by taking symplectic isomorphisms between the lattices and extending them $\R$-linearly. This gives us diffeomorphisms 
$$\overline{\psi}_X:X\to V_\eta/\Lambda_\eta$$
$$\overline{\psi}_Y:Y\to W_\eta/\Gamma_\eta.$$
Let $m:V_\eta\times W_\eta\to\R^{2n}$ be the addition map, and let $H_\eta:=m^{-1}(\Z^{2n})/(\Lambda_\eta\oplus\Gamma_\eta).$
This gives us (via $\overline{\psi}_X\times\overline{\psi}_Y$) a finite subgroup $H\leq X\times Y$. It is not hard to see that if $D=(d_1,\ldots,d_u)$, then $|H_\eta|=|H|=(d_1\cdots d_u)^2$. Moreover, $H\cap\{0\}\times Y=H\cap X\times\{0\}=\{(0,0)\}$, and so $H$ is a maximal isotropic subgroup for $\mathcal{L}_X\boxtimes \mathcal{L}_Y$. By descent theory for abelian varieties, $\mathcal{L}_X\boxtimes\mathcal{L}_Y$ descends to a principal polarization $\mathcal{L}$ with alternating form $E_A$ on $A=(X\times Y)/H$. Let $A=U/\Lambda$, and let $\psi_U:U\to\mathbb{R}^{2n}$ be the $\R$-linear isomorphism that makes the diagram\\

\centerline{\xymatrix{V\times W\ar[r]
\ar[d]_{\psi_X\times\psi_Y}&U\ar[d]^{\psi_U}\\V_\eta\times W_\eta\ar[r]^m&\R^{2n}
}}

\vspace{0.5cm}

\noindent commute. Since the map $V\times W\to U$ is an isomorphism and $E_X\boxtimes E_Y=\psi_X^*E|_{V_\eta}\boxtimes\psi_Y^*E|_{W_\eta}$, we see that $\psi_U^*E=E_A$. 

Let $N_X$ (resp. $N_Y$) be the rational representations of the norm endomorphism of $X$ (resp. $Y$) on $A$ with respect to the symplectic basis induced by $\psi_U$ and the canonical basis on $\R^{2n}$. We see that 
$$N_X+N_Y=dI=\psi_U^{-1}N_\eta\psi_U+\psi_U^{-1}N_{d\theta-\eta}\psi_U.$$
Let $x\in U$ be in the image of $V\hookrightarrow U$ and $y$ in the image of $W$. Then $dx=N_X(x)=\psi_U^{-1}N_\eta\psi_U(x)$ and $N_X(y)=\psi_U^{-1}N_\eta\psi_U(y)=0$. This shows that $N_X=\psi_U^{-1}N_\eta\psi_U$. Therefore the numerical divisor class of $X$ is precisely $\eta$. Since $(X,\mathcal{L_X})$ and $(Y,\mathcal{L}_Y)$ were chosen arbitrarily, we see that the image of the map from Theorem~\ref{moduli} lies in $\mathcal{A}_n(\eta)$.
 
\end{proof}

\begin{proposition}\label{Bijection}
The function 
$$B(n,u,d)/\mbox{Sp}(2n,\Z)\to\{(d_1,\ldots,d_{u-1},d)\in\N^{u}:d_i\mid d_{i+1}, d_{u-1}\mid d\}$$ 
that sends $\eta$ to its type is bijective.
\end{proposition}

\begin{proof}
The function is surjective, since given any type, there is a ppav that contains an abelian subvariety of the same type. What is left to show is that if $\eta,\omega\in B(n,u,d)$ are of the same type, then they differ by the action of a symplectic matrix. But this is trivial since by the previous two lemmas, $\mathcal{A}_n(\eta)=\mathcal{A}_n(\omega)$.
\end{proof}

\begin{corollary}
For every $\eta\in B(n,u,d)$ of type $D$,
$$\mathcal{A}_n(\eta)\simeq\left\{\begin{array}{ll}(\mathcal{A}_u(D)\times\mathcal{A}_{n-u}(\tilde{D}))/\mbox{Sp}(D)& \mbox{if }u<n/2\\(\mathcal{A}_{n/2}(D)\times\mathcal{A}_{n/2}(D))/\Z/2\Z\ltimes\mbox{Sp}(D)&\mbox{if }n\mbox{ is even and }u=n/2.
\end{array}\right.$$
\end{corollary}

\begin{example}
In \cite{Humbert}, Humbert found equations for the moduli space of abelian surfaces that contain an elliptic curve of exponent $m$ (which in his terminology are abelian surfaces that satisfy a singular relation of discriminant $m^2$). Indeed, he showed that the principally polarized abelian surface associated to 
$$\left(\begin{array}{cc}\tau_1&\tau_2\\\tau_2&\tau_3\end{array}\right)$$
contains an elliptic curve of exponent $m$ if and only if there exists a primitive vector $(a,b,c,d,e)\in\Z^5$ such that
$$b^2-4(ac+de)=m^2$$
and
$$a\tau_1+b\tau_2+c\tau_3+d(\tau_1\tau_3-\tau_2^2)+e=0.$$
If we take the differential form
$$\eta_m:=-d dx_1\wedge dx_2+\frac{b-m}{2}dx_1\wedge dx_3-a dx_1\wedge dx_4+c dx_2\wedge dx_3-\frac{b+m}{2}dx_2\wedge dx_4+e dx_3\wedge dx_4,$$
we see that these conditions are satisfied if and only if $\eta_m\wedge\eta_m=0$ and $\eta_m\wedge dz_1\wedge dz_2=0$. Moreover, it is easy to see that $\eta_m$ has degree $m$. Therefore, the Humbert surface of discriminant $m^2$ is just the moduli space $\mathcal{A}_2(\eta_m)$. This was explained in the language of numerical classes in Kani \cite{Kani}. By Proposition~\ref{Bijection}, all matrices of the form
$$\left(\begin{array}{cccc}0&-d&\frac{b-m}{2}&-a\\d&0&c&-\frac{b+m}{2}\\\frac{m-b}{2}&-c&0&e\\a&\frac{b+m}{2}&-e&0\end{array}\right)$$
for $(a,b,c,d,e)\in\Z^5$ primitive and $b^2-4(ac+de)=m^2$ are equivalent under the action of $\mbox{Sp}(4,\Z)$. Therefore if we take, for example, $(a,b,c,d,e)=(1,m,0,0,0)$, we get that the moduli space of principally polarized abelian surfaces that contain an elliptic curve of exponent $m$ is given by the projection of
$$\left\{\left(\begin{array}{cc}\tau_1&\tau_2\\\tau_2&\tau_3\end{array}\right)\in\mathbb{H}_2:\tau_1+m\tau_2=0\right\}$$
to $\mathcal{A}_2$.

In higher dimension, we can take the differential form 
$$\eta_{m,n}:=-mdx_{1}\wedge dx_{n+1}+dx_2\wedge dx_{n+1}.$$
We get that $\eta_{m,n}\in B(n,1,m)$ and so the moduli space of all ppavs of dimension $n$ that contain an elliptic curve of exponent $m$ is given by the projection of 
$$\{(\tau_{ij})_{i,j}\in\mathbb{H}_n:\tau_{11}+m\tau_{12}=0,\tau_{1j}=0\mbox{ for }j=3,\ldots,n\}$$ 
to $\mathcal{A}_n$. \qed
\end{example}

\bibliographystyle{alpha}
\bibliography{bibliographypaper2}{}

\end{document}